\documentclass{svproc}

\usepackage{url, graphicx, amsmath, amssymb, cite}

\newtheorem{conj}{Conjecture}

\begin{document}
\mainmatter             

\title{A Note on the Immersion Number of Generalized Mycielski Graphs}

\titlerunning{Immersion Number of Mycielski Graphs}

\author{Karen L. Collins\inst{1} \and Megan E. Heenehan\inst{2} \and Jessica McDonald\inst{3}}

\authorrunning{K. Collins, M. Heenehan, J. McDonald}

\tocauthor{Karen L. Collins, Megan E. Heenehan, and Jessica McDonald}

\institute{Department of Mathematics and Computer Science, Wesleyan University, Middletown, CT 06459, USA,\\ \email{kcollins@wesleyan.edu}
\and
Department of Mathematical Sciences, Eastern Connecticut State University, Willimantic, CT 06226, USA,\\
\email{heenehanm@easternct.edu}
\and
Department of Mathematics and Statistics, Auburn University, Auburn, AL 36849, USA,\\
\email{mcdonald@auburn.edu}
}

\maketitle

\begin{abstract}
The immersion number of a graph $G$, denoted ${\rm im}(G)$, is the largest $t$ such that $G$ has a $K_t$-immersion. In this note we are interested in determining the immersion number of the $m$-Mycielskian of $G$, denoted $\mu_m(G)$. Given the immersion number of $G$ we provide a lower bound for ${\rm im}(\mu_m(G))$. To do this we introduce the ``distinct neighbor property" of immersions. We also include examples of classes of graphs where ${\rm im}(\mu_m(G))$ exceeds the lower bound. We conclude with a conjecture about ${\rm im}(\mu_m(K_t))$.
\end{abstract}

\section{Introduction}
A pair of adjacent edges $uv$ and $vw$ in a graph are \emph{split off} (or \emph{lifted}) from their common vertex $v$ by deleting the edges $uv$ and $vw$, and adding the edge $uw$. Given graphs $G$ and $H$, we say that $G$ has a \emph{$H$-immersion} if a graph isomorphic to $H$ can be obtained from a subgraph of $G$ by splitting off pairs of edges, and removing isolated vertices. Immersions have gained considerable interest in the last number of years (see eg. \cite{GM, GPRTW, WoWo}). In \cite{CHM} we define the immersion number of $G$, denoted ${\rm im}(G)$, to be the largest $t$ for which $G$ has a $K_t$-immersion.  Abu-Khzam and Langston \cite{AKL} have conjectured that for a graph $G$, if the chromatic number of $G$ is  $t$, then ${\rm im}(G)\geq t$. 

In this note we focus on the immersion number of Mycielski graphs and generalized Mycielski graphs.  Mycielski \cite{My} proved that if $G$ is triangle-free and $\chi(G)=t$, then $\mu(G)$, the Mycielskian of $G$ (whose definition appears in the next section), is triangle-free and $\chi(\mu(G))=t+1$. Hence Mycielski graphs provide examples of triangle-free graphs with arbitrarily high chromatic number. The Mycielski construction was generalized by Stiebitz \cite{St} (see also \cite{SSt}) and independently by Van Ngoc \cite{VN} (see also \cite{VNT}). The general construction was also described as the ``cone over $G$" by Tardif \cite{Tar}. While the chromatic number of the generalized Mycielski graphs does not necessarily increase when the construction is applied \cite{MS}, Stiebitz \cite{St} proved the chromatic number does increase if the construction is applied to a specific class of graphs. We will show that the immersion number of Mycielski graphs and generalized Mycielski graphs increases by at least 1, showing that these graphs behave as expected with respect to the Abu-Khzam--Langston Conjecture.

We begin, in Section~\ref{def}, by providing an alternate definition of immersion along with definitions for Mycielski graphs and generalized Mycielski graphs. In Section~\ref{result}, we define the ``distinct neighbor property" of immersions and prove that if a graph has a $K_t$-immersion, then it has a $K_t$-immersion with the distinct neighbor property. This is then used to prove that if the immersion number of $G$ is $t$, then the immersion number of the Mycieskian (or generalized Mycielskian) of a graph is at least $t+1$. In Section~\ref{ex}, we provide several examples of graphs where this increase is more than one. We conclude with a conjecture about the immersion number of the generalized Mycielskian of complete graphs.

\section{Definitions}\label{def}
It will be useful to keep the following alternate definition of immersion in mind. 

\begin{definition}Given graphs $G$ and $H$, $G$ has an $H$-\emph{immersion} if there is a one-to-one function $\phi: V(H)\to V(G)$ such that for each edge $uv\in E(H)$, there is a path $P_{uv}$ in $G$ joining vertices $\phi(u)$ and $\phi(v)$, and the paths $P_{uv}$ are pairwise edge-disjoint for all $uv\in E(H)$. In this context we call the vertices of $\{\phi(v): v\in V(H)\}$ the \emph{terminals} of the $H$-immersion, and we call internal vertices of the paths $\{P_{uv}: uv\in E(H)\}$ the \emph{pegs} of the $H$-immersion. \end{definition}

\noindent In this note, we are interested in clique immersions in Mycielski graphs.

\begin{definition}\label{defMy} {\rm \cite{My}}
Given a graph $G$ the \emph{Mycielski graph} (or Mycielskian of a graph), denoted $\mu(G)$, is defined to be the graph with vertex set $(V(G)\times\{0, 1\})\cup\{w\}$, and edges $(u,0)-(v,0)$ and $(u,0)-(v,1)$ for all $uv\in E(G)$, and edges $(u,1)-w$ for all $u\in V(G)$.
\end{definition}
The notation in Definition~\ref{defMy} follows~\cite{MS}. Figure~\ref{M(C_5)} shows $\mu(C_5)$, as an example.

\begin{figure}
\begin{center}
\includegraphics[width=2in]{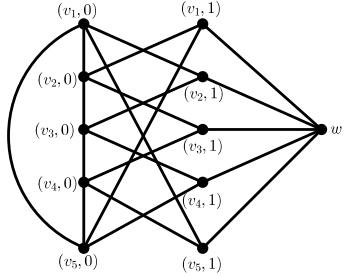}
\caption{The Mycielski graph $\mu(C_5)=\mu_2(C_5)$, where consecutive vertices in the $C_5$ have been labeled $v_1, v_1, \ldots, v_5$.}
\label{M(C_5)}
\end{center}
\end{figure}

The Mycielski construction can be generalized to the \emph{$m$-Mycielskian of $G$}, denoted $\mu_m(G)$, as follows. 

\begin{definition}\label{GenMy} {\rm\cite{SSt, St, VN, VNT}} Given a graph $G$ and an integer $m\geq 1$, the \emph{$m$-Mycielskian of $G$}, denoted $\mu_m(G)$, is defined to be the graph with vertex set $(V(G)\times\{0, 1, \ldots, m-1\})\cup\{w\}$, and edges $(u,0)-(v,0)$ and $(u,i)-(v,i+1)$ for all $uv\in E(G)$, and edges $(u,m-1)-w$ for all $u\in V(G)$.
\end{definition}
The notation in Definition~\ref{GenMy} again follows~\cite{MS}. Notice that $\mu_1(G)$ is obtained from a copy of $G$ with an additional vertex adjacent to all vertices in $G$ and $\mu_2(G)$ is the Mycielskian of $G$, that is $\mu(G)=\mu_2(G)$. This construction is also known as the \emph{cone over $G$} \cite{Tar}. When described as the cone over $G$  we use a direct product of graphs. We let $\mathbb{P}_m$ be a path on $m+1$ vertices with a loop at one end. The \emph{$m$-Mycielskian of $G$}, denoted $\mu_m(G)$, is obtained from $G\times \mathbb{P}_m$ (where $\times$ is the direct product) by collapsing all vertices whose second coordinate is $m$ to a single vertex labeled $(*,m)$ ($w$ in Definition~\ref{GenMy}) \cite{Tar}.
 Figure~\ref{M_4(C_5)} shows $\mu_4(C_5)$, as an example.

\begin{figure}
\begin{center}
\includegraphics[width=3.5in]{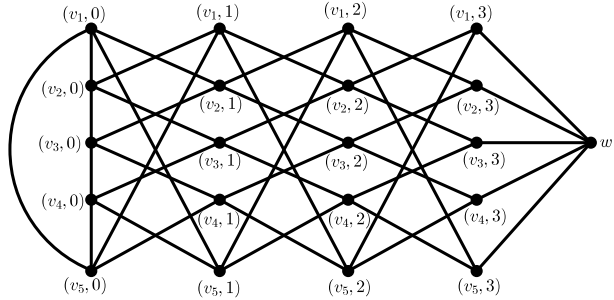}
\caption{The $4$-Mycielskian of $C_5$, $\mu_4(C_5)$, where consecutive vertices in the $C_5$ have been labeled $v_1, v_1, \ldots, v_5$.}
\label{M_4(C_5)}
\end{center}
\end{figure}

Notice that $\chi(\mu_1(G))=\chi(G)+1$ and $\chi(\mu_2(G))=\chi(G)+1$ \cite{My}. Stiebitz proved that when $M(3)$ is the set of all $3$-color-critical graphs and for $k\geq 3$ and $M(k+1)=\{\mu_m(G)~|~G\in M(k) \textrm{ and } r\geq 1\}$, then $M(k)$ is $k$-color-critical for all $k\geq3$ \cite{St}. However for all $t\geq 4$, there exists a graph $G$ such that $\chi(G)=\chi(\mu_3(G))=t$ \cite{Tar}. The smallest example of this is the circulant graph formed by placing $7$ vertices around a circle, equally spaced, each vertex is adjacent to its nearest two vertices in each direction around the circle. This graph has chromatic number $4$ as does the $3$-Mycielskian of the graph \cite{Tar}.

Our main result is Theorem~\ref{Myc}, appearing in the next section, which states that if ${\rm im}(G)=t$ and $m\geq 1$, then ${\rm im}(\mu_m(G))\geq t+1$. This result confirms the Abu-Khzam--Langston Conjecture for this class of graphs. Moreover, given the special nature of generalized Mycielski graphs, as an immediate corollary  we get that there exist graphs with fixed clique number and arbitrarily high chromatic number that satisfy the conjecture.

In the next section we prove Theorem~\ref{Myc}. This is accomplished by first defining the ``distinct neighbor property" of an immersion and proving every graph with a $K_t$-immersion contains a $K_t$-immersion with the distinct neighbor property.

\section{Results}\label{result} 

As the main element of proof of Theorem~\ref{Myc}, we present a lemma about immersions with the ``distinct neighbor property.'' Propositon~\ref{dnp} proves that if $G$ has a $K_t$-immersion, then it has a $K_t$-immersion in which each terminal has a distinct neighbor. This may be of independent interest as immersions require edges on paths be distinct, but not vertices. 

\begin{definition} Let $G$ be a graph having a $K_t$-immersion, with terminals $v_1, \ldots, v_t$. If, for each $i\in\{1, 2, \ldots, t\}$, it is possible to choose a distinct neighbor $v_{f(i)}$ of $v_i$ in $G$ (i.e. such that $\{v_{f(i)}: 1\leq i\leq t\}$ is a set of $t$ vertices), then we say that the immersion has the \emph{distinct neighbor property}.\end{definition}

\begin{proposition}\label{dnp} Let $G$ be a graph that has a $K_t$-immersion. Then $G$ has a $K_t$-immersion with the distinct neighbor property.
\end{proposition}

\begin{proof} Let $A=\{a_1\,\ldots, a_t\}$ be the set of terminals of a $K_t$-immersion in $G$ that does not have the distinct neighbor property. Let $B=\{b_1, b_2, \ldots\}$ be the set of all vertices in $G$ which are neighbors of at least one vertex in $A$ (this definition allows for $A\cap B\neq \emptyset$, although we will soon see that this does not actually occur). Define $H$ to be the bipartite graph with bipartition $(A, B)$  where $a_ib_j\in E(H)$ if and only if $a_ib_j\in E(G)$. Since our immersion does not have the distinct neighbor property, $H$ does not have a matching saturating $A$. Hence by Hall's Theorem \cite{Hall}, there exists $S\subseteq A$ such that $|N_H(S)|<|S|$. Since $A$ is the set of terminals of a $K_t$-immersion, and $G$ is simple, each vertex in $A$ has at least $t-1$ distinct neighbors in $B$. Hence it must be the case that $N_H(a_i)=B$ for all $i\in\{1, 2, \ldots, t\}$ and $|B|=t-1$. In particular, this means that $A\cap B=\emptyset$ and hence that $H$ is a subgraph of $G$.

We now describe a new $K_t$-immersion in $G$ with terminals \[b_1, b_2, \ldots, b_{t-1}, a_t. \] First note that $a_t$ is adjacent to all of the other terminals in the list. For the other paths required, we will show that in fact $H\setminus a_t$ has a $K_{t-1}$-immersion using $b_1, \ldots, b_{t-1}$ as terminals.

Consider an auxiliary copy of $K_{t-1}$ that has been properly edge-colored with at most $t-1$ colors (possible by Vizing's Theorem \cite{V2}). Label the vertices of the $K_{t-1}$ with $b_1,\ldots, b_{t-1}$ and label the colors of the edges with $a_1, a_2, \ldots, a_{t-1}$. For each pair $b_i, b_j$, the color of the edge between them in this auxiliary graph (say $a_k$) corresponds the path $b_i-a_k-b_j$ we should pick to join them in the $K_t$-immersion we are building in $G$. Since the edge-coloring is proper, no edge of $G$ gets used twice in this assignment of paths, and we have defined a $K_t$-immersion. 

Using this new $K_t$-immersion in $G$ we define a new $H$ in the same way as at the beginning of the proof. We now see there is a matching saturating $A$ and therefore this new $K_t$-immersion has the distinct neighbor property. For example, for each terminal $b_i$ we choose as its distinct neighbor $a_i$ (for $i\in\{1,2,\ldots, t-1\}$) and for $a_t$ we choose $b_1$ as its distinct neighbor (really any $b_i$ works since $a_t$ is adjacent to $b_i$). $\Box$
\end{proof}

We now use Proposition~\ref{dnp} to prove Theorem~\ref{Myc}.

\begin{theorem}\label{Myc} Let $m\geq 1$ be an integer and suppose that ${\rm im}(G)=t$. Then ${\rm im}(\mu_m(G))\geq t+1$.
\end{theorem}

\begin{proof} By Proposition \ref{dnp}, we can choose a $K_t$-immersion in $G$ with the distinct neighbor property. Let $v_1, \ldots, v_t$ be the terminals of the immersion, and let $v_{f(1)}, \ldots, v_{f(t)}$ be their distinct neighbors. 

When $m=1$ we use the immersion in $G$ with the additional terminal $w$ which is adjacent to all vertices in $G$.

When $m=2$, to form a $K_{t+1}$-immersion in $\mu_2(G)$, we take $(v_1,0), (v_2, 0), \ldots,$ $(v_t, 0)$ and $w$ as our terminals. The first $t$ of these have the required paths joining them in $G$, and the vertex $w$ is joined to each $(v_i, 0)$ via the path $(v_i,0)-(v_{f(i)}, 1)-w$.  These paths are edge-disjoint since $v_f(i)$ is $v_i$'s distinct neighbor.

For $m>2$, we again take $(v_1,0), (v_2, 0), \ldots, (v_t, 0)$ and $w$ as our terminals. We use the $K_t$-immersion in $G$ to connect the first $t$ terminals. To join $w$ to the other terminals we alternate between $v_i$ and its distinct neighbor $v_{f(i)}$ (that is $(v_i,0)-(v_{f(i)}, 1)-(v_i, 2)-(v_{f(i)}, 3)-\dots$) until we come to the vertex with second coordinate $m-1$ which is adjacent to $w$. This completes the $K_{t+1}$-immersion in $\mu_m(G)$. $\Box$
\end{proof}

In the next section we provide examples of classes of graphs where the bound of Theorem~\ref{Myc} is best possible and examples of classes of graphs that exceed the bound of Theorem~\ref{Myc}. 

\section{Examples}\label{ex}
While chromatic number may not increase by more than one when the generalized Mycielski construction is applied, the immersion number of certain classes of graphs may increase more drastically. When $m$ increases the paths in $\mu_m(G)$ get longer and the degree of many vertices doubles. In some cases this allows the immersion number to increase by more than one. We explore this idea by examining several classes of graphs.

To calculate the immersion numbers of specific Mycielski (and generalized Mycielski) graphs we must consider the degrees of the vertices in the graph and the number of vertices of each degree. In order to have an immersion of $K_t$ there must be at least $t$ vertices of degree $t-1$ and the immersion number of any graph must be less than or equal to the maximum degree plus one. 

The bound of Theorem~\ref{Myc} is best possible when $m=1$ since, when forming $\mu_1(G)$, one vertex is added that is adjacent to all vertices in $G$. The bound is also best possible when $m=2$ and $G$ is a complete graph.

\begin{proposition}\label{MycK_t}
${\rm im}(\mu_2(K_t))=t+1$
\end{proposition}

\begin{proof}
The immersion number of $K_t$ is $t$, thus by Theorem~\ref{Myc} ${\rm im}(\mu_2(K_t))\geq t+1$. The graph $\mu_2(K_t)$ has $t$ vertices of degree $2t-2$ and $t+1$ vertices of degree $t$, therefore ${\rm im}(\mu_2(K_t))\leq t+1$. Thus ${\rm im}(\mu_2(K_t))=t+1$.
$\Box$
\end{proof}

The following provides examples where the immersion number is greater than the bound of Theorem~\ref{Myc}. We begin by considering the path on $n$ vertices, $P_n$, with the knowledge that ${\rm im}(P_n)=2$.

\begin{proposition}\label{MyP_5}
 $ {\rm im}(\mu_m(P_5))=
  \begin{cases}
    4 & \text{for } m = 1, 2 \\
    5 & \text{for } m \geq 3 
  \end{cases}$
\end{proposition}

\begin{proof}
Label consecutive vertices in $P_5$ as $1, 2, \ldots, 5$ so that vertex $1$ has degree $1$ and consider $\mu_m(P_5)$.
When $m=1$ or $2$ there are not enough vertices of degree $4$ for $\mu_m(P_5)$ to have a $K_5$-immersion, therefore ${\rm im}(\mu_m(P_5))\leq 4$. Immersions of $K_4$ can easily be found in these cases.

When $m\geq 3$, there are $2m$ vertices of degree $2$, $m$ vertices of degree $3$, $3(m-1)$ vertices of degree $4$, and one vertex of degree $5$, therefore ${\rm im}(\mu_m(P_5))\leq 5$. We now provide a construction for a $K_5$-immersion.

As terminals of our $K_5$-immersion we use vertices $(2,0), (3,0), (4,0), (2,1),$ and $(4,1)$. Notice that $(3,0)$ is adjacent to all of the other terminals. To connect the remaining pairs of terminals we use the paths:
$$(2,0)-(3,1)-(4,0)\hspace{.2in} (2, 0)-(1,0)-(2,1)$$ 
$$(2, 1)-(3,2)-(4,1) \hspace{.2in} (4, 0)-(5,0)-(4,1)$$
$$(2,0)-(1,1)-(2,2)-(1,3)-\ldots - w - \ldots -(5,4)-(4,3)-(5,2)-(4,1)$$
$$(4,0)-(5,1)-(4,2)-(5,3)- \ldots - w- \ldots - (1,4)-(2,3)-(1,2)-(2,1)$$
This completes a $K_5$-immersion in $\mu_m(P_5)$.
$\Box$
\end{proof}

This result can be generalized as follows.

\begin{corollary}\label{GenMycP_n}
If $m\geq 3$ and $n\geq5$, then ${\rm im}(\mu_m(P_n))=5$.
\end{corollary}

\begin{proof}
Let $m\geq 3$ and $n\geq 5$. In $\mu_m(P_n)$ there are $2m$ vertices of degree $2$, $n-2$ vertices of degree $3$, $(m-1)(n-2)$ vertices of degree $4$, and one vertex of degree $n$. Therefore  ${\rm im}(\mu_m(P_n))\leq5$. We see that $\mu_m(P_5)$ is a subgraph of $\mu_m(P_n)$, thus by Proposition~\ref{MyP_5}  ${\rm im}(\mu_m(P_n))\geq5$. We have shown  ${\rm im}(\mu_m(P_n))=5$.
$\Box$
\end{proof}

As another example where the immersion number exceeds the bound of Theorem~\ref{Myc} we consider the cycle on $n$ vertices, $C_n$. Note that ${\rm im}(C_n)=3$.

\begin{proposition} If $n\geq 5$, then 
$ {\rm im}(\mu_m(C_n))=
  \begin{cases}
    4 & \text{for } m = 1\\
    5 & \text{for } m \geq 2 
  \end{cases}$
\end{proposition}

\begin{proof}
Let $n\geq 5$. 

In $\mu_1(C_n)$ there are $n$ vertices of degree $3$ and one vertex of degree $n$, therefore  ${\rm im}(\mu_1(C_n))\leq4$. Notice that $\mu_1(P_5)$ is a subgraph of $\mu_1(C_n)$ thus, by Proposition~\ref{MyP_5}, ${\rm im}(\mu_1(C_n))\geq4$ and we get ${\rm im}(\mu_1(C_n))=4$.

When $m\geq 2$, $\mu_m(C_n)$ has $n$ vertices of degree $3$, $(m-1)n\geq 5$ vertices of degree $4$, and one vertex of degree $n$, therefore ${\rm im}(\mu_m(C_n))\leq5$.  

When $m\geq 3$, we use the fact that $\mu_m(P_n)$ is a subgraph of $\mu_m(C_n)$ therefore, by Proposition~\ref{MyP_5} and Corollary~\ref{GenMycP_n}, ${\rm im}(\mu_m(C_n))\geq5$. Thus for $n\geq 5$ and $m\geq 3$ ${\rm im}(\mu_m(C_n))=5$.

For the case when $m=2$ we will construct a $K_5$-immersion. Label consecutive vertices in $C_n$ $v_1, v_2, \ldots, v_n$. As the terminals of our immersion we use $(v_1, 0), (v_2, 0), (v_3, 0), (v_4, 0)$ and $(v_5, 0)$. No matter the value of $n$ we use edges $(v_1, 0)-(v_2, 0), (v_2, 0)-(v_3, 0), (v_3, 0)-(v_4, 0),$ and $(v_4, 0)-(v_5, 0)$ and the paths $(v_1, 0)-(v_2,1)-(v_3,0)$,  $(v_2, 0)-(v_3,1)-(v_4,0)$, and $(v_3, 0)-(v_4,1)-(v_5,0)$ in our $K_5$-immersion. The remaining paths depend on the value of $n$.

\noindent Case 1: When $n=5$ the remaining paths are
$$(v_1,0)-(v_5,1)-(v_4,0),$$ $$(v_1,0)-(v_5,0),$$ $$(v_2,0)-(v_1,1)-(v_5,0)$$
Case 2: When $n=6$ the remaining paths are
$$(v_1,0)-(v_6,0)-(v_5,1)-(v_4,0),$$ $$(v_1,0)-(v_6,1)-(v_5,0),$$ $$(v_2,0)-(v_1,1)-(v_6,0)-(v_5,0)$$
Case 3: When $n\geq7$ the remaining paths are
$$(v_1,0)-(v_n,1)-w-(v_5,1)-(v_4,0),$$ $$(v_1,0)-(v_n,0)-(v_{n-1},0) -\dots-(v_6,0)-(v_5,0),$$ $$(v_2,0)-(v_1,1)-w-(v_6,1)-(v_5,0)$$
This completes the $K_5$-immersion in $\mu_2(C_n)$. $\Box$
\end{proof}

Some preliminary work leads us to believe that ${\rm im}(\mu_m(K_t))> t+1$ for a large enough $m$. For example, we can show ${\rm im}(\mu_3(K_4))=7$ and ${\rm im}(\mu_4(K_5))=9$ (see Appendix \ref{appendix} for explicit immersions). In $\mu_m(K_t)$ there are $t+1$ vertices of degree $t$ and $(m-1)t$ vertices of degree $2t-2$. When $m>2$ there are enough vertices to potentially have a $K_{2t-1}$-immersion in ${\rm im}(\mu_m(K_t))$ and for $t>2$ we know $t+1< 2t-1$. This means $t+1\leq{\rm im}(\mu_m(K_t))\leq 2t-1$ for $m>2$ and $t>2$. This leads us to make the following conjecture.

\begin{conj} If $m\geq 3$, then ${\rm im}(\mu_m(K_{m+1}))=2m+1$.
\end{conj}

\appendix 

\section{Appendix}\label{appendix}

Here we provide explicit immersions in $\mu_3(K_4)$ and $\mu_4(K_5)$.

\begin{example}
Label the vertices of a $K_4$ $v_1, v_2, \ldots, v_4$. As the terminals of our  $K_7$-immersion in $\mu_3(K_4)$ we use vertices $(v_i, 0)$ for $i=1, 2, 3, 4$ and $(v_j, 1)$ for $j=1, 2, 3$. The terminals $(v_i, 0)$ form a $K_4$ and $(v_i, 0)$ is adjacent to $(v_j, 1)$ when $i\neq j$. We complete the $K_7$-immersion using the following paths.
\begin{align*}
\label{}
    &(v_1,0)-(v_4, 1)-(v_2,2)-(v_1,1) &\hspace{.2in} (v_1,1)-(v_4, 2)-(v_3,1)   \\
    &(v_1,1)-(v_3, 2)-w-(v_1,2)-(v_2,1) &\hspace{.2in}   (v_2,0)-(v_4, 1)-(v_3,2)-(v_2,1)  \\
    &(v_2,1)-(v_4, 2)-w-(v_2,2)-(v_3,1)  & \hspace{.2in}  (v_3, 0)-(v_4, 1)-(v_1,2)-(v_3,1)
\end{align*}
\end{example}

\begin{example}
Label the vertices of a $K_5$ $v_1, v_2, \ldots, v_5$. As the terminals of our  $K_9$-immersion in $\mu_4(K_5)$ we use vertices $(v_i, 0)$ for $i=1, 2, 3, 4, 5$ and $(v_j, 1)$ for $j=1, 2, 3, 4$. The terminals $(v_i, 0)$ form a $K_5$ and $(v_i, 0)$ is adjacent to $(v_j, 1)$ when $i\neq j$. We complete the $K_9$-immersion using the following paths.
\begin{align*}
\label{}
    &(v_1,0)-(v_5, 1)-(v_4,2)-(v_1,1) &\hspace{.2in} (v_2,0)-(v_5,1)-(v_3,2)-(v_2,1)  \\
    &(v_3,0)-(v_5, 1)-(v_2,2)-(v_3,1) &\hspace{.2in} (v_4,0)-(v_5, 1)-(v_1,2)-(v_4,1) \\
    &(v_1,1)-(v_3, 2)-(v_2, 3)-(v_4,2)-(v_2,1) &\hspace{.2in} (v_1,1)-(v_5,2)-(v_3,1)  \\
    & (v_1,1)-(v_2, 2)-(v_4,1)  &\hspace{.2in}  (v_2,1)-(v_1, 2)-(v_3,1)\\
      & (v_3,1)-(v_4, 2)-(v_5,3)-(v_3, 2)-(v_4, 1) & \hspace{.2in} (v_2, 1)-(v_5, 2)-(v_4,1)  
\end{align*}
\end{example}


\begin{thebibliography}{00}

\bibitem{AKL} F. Abu-Khzam and M. Langston,
Graph coloring and the immersion order.
\emph{Lecture Notes in Computer Science} \textbf{2697} (2003), 394-403.

\bibitem{CHM} K. L. Collins, M. E. Heenehan, and J. McDonald,
Clique immersion in graph products, manuscript. (arXiv:1908.10457 [math.CO])
%
\bibitem{GPRTW} A. Giannopoulou, M. Pilipczuk, J. Raymond, D. Thilikos, and M. Wrochna,
Linear Kernals for Edge Deletion Problems to Immersion-Closed Graph Classes.
\emph{SIAM J. Discrete Math.} \textbf{35(1)} (2021) 105-151.
%
\bibitem{GM} M. Guyer and J. McDonald,
On clique immersions in line graphs.
\emph{Discrete Math.} \textbf{343(12)} (2020) 7 pp.

\bibitem{Hall} P. Hall, On Representatives of Subsets. \emph{J. London Math. Soc.} \textbf{10} (1935), 26–-30.

\bibitem{MS} T. M\"uller and M. Stehl\'{\i}k,
Generalised Mycielski graphs and the Borsuk-Ulam theorem,
\emph{Electron. J. Comb.} \textbf{26(4)} (2019).

\bibitem{My} J. Mycielski, Sur le coloriage des graphes, \emph{Colloq. Math.} \textbf{3} (1955) 161\^a-162.

\bibitem{SSt} H. Sachs and M. Stiebitz,
On constructive methods in the theory of colour-critical graphs,
\emph{Discrete Math.} \textbf{74(1-2)} (1989) 201-226.

\bibitem{St} M. Stiebitz,
Beitr\"age zur Theorie der f\"arbungskritischen Graphen,
\emph{Dissertationsschrift} \textbf{(B) TH} Ilmenau (1985).

\bibitem{Tar} C. Tardif,
Fractional Chromatic Numbers of Cones Over Graphs,
{\it J. Graph Th.} {\bf 38} (2001) 87-94.

\bibitem{VN} N. Van Ngoc,
On graph colourings,
PhD thesis, Hungarian Academy of Sciences, 1987.

\bibitem{VNT} N. Van Ngoc and Zs. Tuza,
$4$-chromatic graphs with large odd girth,
\emph{Discrete Math.} \textbf{138(1-3)} (1995) 387-392.

\bibitem{V2} V. Vizing, On an estimate of the chromatic class of a $p$-graph, \emph{Diskret. Analiz} \textbf{3} (1964), 25-30.

\bibitem{WoWo} P. Wollan and D. Wood, Nonrepetitive colorings of graphs excluding a fixed immersion or topological minor, \emph{J. Graph Th.} \textbf{91(3)} (2019), 259 - 266.

\end{thebibliography}
\end{document}